\documentclass{amsart}
\usepackage{longtable}
\usepackage{amsthm}
\usepackage{amsmath}
\usepackage{amssymb}
\usepackage{enumerate}
\usepackage{array}
\usepackage{enumerate}
\usepackage{bm}
\usepackage{hyperref}
\usepackage{xcolor}
\hypersetup{
	colorlinks=true,
	linkcolor=cyan,
	citecolor=cyan,
}

\newtheorem{theorem}{Theorem}[section]
\newtheorem{lemma}{Lemma}[section]

\newtheorem{rk}{Remark}[section]
\newtheorem{de}{Definition}[section]
\newtheorem{eg}{Example}[section]
\newcommand{\C}{\mathcal{C}}

\title{Primitive Normal Values of Rational Functions over Finite Fields}

\author{Avnish K. Sharma}
\address{Department of Mathematics, University of Delhi, New Delhi-110007, India}
\email{avkush94@gmail.com}

\author{Mamta Rani}
\address{Department of Mathematics, University of Delhi, New Delhi-110007, India}
\email{mamta11singla@gmail.com}

\author{Sharwan K. Tiwari$^*$}
\address{Scientific Analysis Group, Defence Research $\&$ Development Organization, Metcalfe House, Delhi-110054, India}
\email{shrawant@gmail.com}
\thanks{$^*$ Corresponding author}

\keywords{Finite fields; Primitive elements; Normal elements; Additive and Multiplicative characters.}

\subjclass{}
\begin{document}
	\maketitle
	\sloppy
	
%-------------Abstract-----------------%
%--------------------------------------%

	    \begin{abstract}
	     In this paper, we consider rational functions $f$ with some minor restrictions over the finite field $\mathbb{F}_{q^n},$ where $q=p^k$ for some prime $p$ and positive integer $k$. We establish a sufficient condition for the existence of a pair $(\alpha,f(\alpha))$ of primitive normal elements in $\mathbb{F}_{q^n}$ over $\mathbb{F}_{q}.$ Moreover, for $q=2^k$ and rational functions $f$ with quadratic numerators and denominators, we explicitly find that there are at most $55$ finite fields $\mathbb{F}_{q^n}$ in which such a pair $(\alpha,f(\alpha))$ of primitive normal elements may not exist.
         \end{abstract}

%-------------Introduction--------------%---------------------------------------
%---------------------------------------%
	\section{Introduction}\label{Sec1}
		Let $q=p^k$ for some prime $p$ and a positive integer $k$, and $\mathbb{F}_{q^n}$ be the extension of finite field $\mathbb{F}_{q}$ of degree $n.$ An element $\alpha\in \mathbb{F}_{q^n}$ is called a {\it primitive element}, if it is a generator of the multiplicative cyclic group $\mathbb{F}_{q^n}^*$, and an element $\beta\in \mathbb{F}_{q^n}$ is called a {\it normal element} over $\mathbb{F}_{q}$, if $\{\beta, \beta^q,\beta^{q^2},\ldots,\beta^{q^{n-1}}\}$ forms a basis (called {\it normal basis}) of $\mathbb{F}_{q^n}$ over $\mathbb{F}_q$. An element of a finite field $\mathbb{F}_{q^n}$ is called a {\it primitive normal element}, if it is both primitive and normal over $\mathbb{F}_{q}$. In the computation theory, normal elements play an important role for the efficient  implementation of finite field arithmetic such as multiplication and exponentiation of field elements \cite{mullin,Omura,Onyszchuk}.
		Various applications of primitive and normal elements can be seen in cryptography, coding theory and signal processing \cite{diffie,blum,Agnew}.

		Carlitz \cite{CarlPrim,CarlSome} discussed the problem of the existence of primitive normal elements over finite fields and showed that such elements exist in sufficiently large fields $\mathbb{F}_{q^n}$ over $\mathbb{F}_q.$ Also, the existence of such elements over prime fields was proved by  Davenport \cite{Daven}. Later, Lenstra and Schoof \cite{Lenstra} completely settled the problem and gave the following result.
		\begin{theorem}
			For every prime power $q>1$ and every positive integer $n$ there exists a primitive normal basis of $\mathbb{F}_{q^n}$ over $\mathbb{F}_q.$
		\end{theorem} 
		  We know that the primitivity of an element is preserved under the inverse in finite fields, but what can we say about the normality of an element? Tian and Qi \cite{TianQi} answered this question partially by proving that for $n\geq32,$ $\mathbb{F}_{q^n}$ contains an element $\alpha$ such that both $\alpha$ and $\alpha^{-1}$ are normal elements over $\mathbb{F}_{q}.$ Later, Cohen and Huczynska \cite{CoHuc} completely resolved this question by proving the following result.
		  
		  \begin{theorem}\cite[Theorem 1.4]{CoHuc}(Strong Primitive Normal Basis Theorem) There exists a primitive normal element $\alpha \in \mathbb{F}_{q^n}$; $n\geq 2$, such that $\alpha^{-1}$ is also a primitive normal over $\mathbb{F}_{q}$ unless $(q,n)$ is one of the pair $(2,3), (2,4),(3,4),(4,3),(5,4).$ 
		  \end{theorem}
	      We define a pair $(\alpha,\beta)$ to be a {\it primitive pair} in $\mathbb{F}_{q^n}$, if both $\alpha$ and $\beta$ are primitive elements in $\mathbb{F}_{q^n}.$ Additionally, if both $\alpha$ and $\beta$ are normal elements over $\mathbb{F}_{q}$, then the pair $(\alpha,\beta)$ is  called a {\it primitive normal pair} in $\mathbb{F}_{q^n}$ over $\mathbb{F}_{q}.$ In this terminology, strong primitive normal basis theorem actually showed the existence of a primitive normal pair $(\alpha,\ \alpha^{-1}).$ It is natural to ask the existence of a primitive (normal) pair $(\alpha,f(\alpha))$ in $\mathbb{F}_{q^n}$ for any relation $f$ over $\mathbb{F}_{q^n}.$ Cohen \cite{CohenConsecutive} showed the existence of primitive pairs $(\alpha,\alpha+1)$ in  $\mathbb{F}_{q}$ for $q\geq3, \ q\not\equiv7(\text{mod} \ 12)$ and $q\not\equiv1(\text{mod} \ 60).$ Cohen \cite{SDEven} also proved the existence of primitive pairs $(\alpha,\alpha+\alpha^{-1})$ in $\mathbb{F}_{q^n}$ for $q=2^k\geq8$ and $n\geq1$, which was already proved by Wang et al. \cite{Wang} for $q=2^k\geq16$ and for odd $n\geq13.$ Recently, Liao et al. \cite{Liao} generalized this result to the case where $q$ is any prime power. Kapetanakis \cite{KapeNBT,Kape} contributed in this direction by proving the existence of primitive normal pairs $\big(\alpha,\frac{a\alpha+b}{c\alpha+d}\big)$ in $\mathbb{F}_{q^n}$ over $\mathbb{F}_{q}$ for $q\geq 23$ and $n\geq17.$ Anju and R.K. Sharma \cite{ARKS} proved the existence of primitive pairs $(\alpha,\alpha^2+\alpha+1)$ in $\mathbb{F}_{q^n}$ and further showed the existence of  primitive pairs $(\alpha,\alpha^2+\alpha+1)$ where $\alpha$ is a normal element in $\mathbb{F}_{q^n}$ over $\mathbb{F}_{q}.$ Booker et al. \cite{Booker}  generalized this result for arbitrary quadratic polynomials $f$ in $\mathbb{F}_{q}[x]$ with non-zero discriminant and gave a small list of genuine exceptions.
	      	
		  In \cite{AnjuPN,Awasthi}, R. K. Sharma et al. discussed the existence of primitive pairs $\Big(\alpha,f(\alpha)=\frac{a\alpha^2+b\alpha+c}{d\alpha^2+e\alpha+f}\Big)$ over finite fields $\mathbb{F}_{2^k}$ and proved the following result.
		  \begin{theorem}\label{T1.3}
		  	Let $f=\dfrac{ax^2+bx+c}{dx^2+ex+f} \in \mathbb{F}_{2^k}(x),$ where $a,d$ are non-zero and  $(dx^2+ex+f)\nmid (ax^2+bx+c)$. Then there always exists a primitive element $\alpha \in \mathbb{F}_{2^k}$ such that $f(\alpha)$ is also a primitive element in $\mathbb{F}_{2^k}$ for all $k$ except $k \in \{1,2,4,6,8,9,10,12\}.$
		  \end{theorem}
	       For the completeness, they provided counter examples for $k=1,2,4$ and conjectured that the above theorem holds for $k=6,8,9,10,12.$ 
	       In section \ref{Sec3} of this paper, we show that for the case $k=3,$ above theorem does not hold by providing a counter example. 
	       	       
{Recently, S.D. Cohen et al. \cite{hari} generalized the above theorem for the rational functions $f=\frac{f_1}{f_2}\in\mathbb{F}_{q}(x)$ that are not of the form $cx^jh^d$ for any $h\in \mathbb{F}_{q}(x),$ $c\in \mathbb{F}_q^*,\ j\in \mathbb{Z},\ d>1$ such that $d\mid q-1.$ In \cite{CarNeu}, Carvalho et al. also generalized the above theorem for a family of rational functions $f=\frac{f_1}{f_2}\in \mathbb{F}_q(x)$ such that there exists at least one monic irreducible polynomial $g\in \mathbb{F}_q[x]$ other than $x$ and a positive integer $n$ with gcd$(n,q-1)=1$, $g^n|f_1f_2$ but $g^{n+1}\nmid f_1f_2.$ Furthermore, Carvalho et al. \cite{carvalho2} established a sufficient condition to prove the existence of a primitive normal element $\alpha \in \mathbb{F}_{q^n}$ over $\mathbb{F}_{q}$  such that $f(\alpha)$ is a primitive element (but not necessarily normal) for the same family. We observed that, any rational function from the latter family cannot be expressed as $cx^jh^d$ for any $h\in \mathbb{F}_{q}(x),$ $c\in \mathbb{F}_q^*,\ j\in \mathbb{Z},\ d>1$ such that $d\mid q-1$, that means each rational function of the latter family belongs to the S.D. Cohen's family, otherwise gcd$(n,q-1)\neq1.$ But the Carvalho's family is a proper subset of S.D. Cohen's family. For example, the rational function $\frac{(x-1)^2}{(x+1)^3}\in \mathbb{F}_{19}(x)$ belongs to the S.D. Cohen's family but not in the Carvalho's family.  

In this paper, we consider primitive elements $\alpha\in \mathbb{F}_{q^n}$ which are also {\it normal} elements over $\mathbb{F}_q$. We provide a sufficient condition for the  existence of a primitive normal element $\alpha \in \mathbb{F}_{q^n}$ such that $f(\alpha)$ is also a primitive normal element in $\mathbb{F}_{q^n}$ for almost all functions $f$ belongs to the S.D. Cohen's family (defined in Section \ref{Sec4}).
	         
		We organize this paper as follows: For the convenience of the reader, we recall some definitions and basic results in Section \ref{Sec2}.  In Section \ref{Sec3}, we show that the case $k=3$ in Theorem \ref{T1.3} is also a genuine exception. In Section \ref{Sec4}, we establish a sufficient condition to show the existence of primitive normal pairs $(\alpha, f(\alpha))$ for every rational function $f$ with some minor exceptions. In Section \ref{Sec5}, an application of this sufficient condition is studied for every rational function $f$ over the fields of even characteristic, where degrees of numerator and denominator of $f$ are $2$. Finally, Section \ref{Sec6} concludes the paper.  

%--------------Section2-----------------%
%-----------Preliminaries---------------%
\section{Prerequisites}\label{Sec2}
In this section, we recall some definitions, lemmas and notations that will be used in proving a sufficient condition for the existence of a primitive normal pair and in finding the pairs $(q,n)$ for which such a primitive normal pair may not exists in $\mathbb{F}_{q^n}$ over $\mathbb{F}_q$.
\begin{de}[\bf Character]
	Let $G$ be a finite abelian group and $T$ be the multiplicative group of complex numbers of modulus 1. A homomorphism $\chi : G \to T$ is called a character of $G.$ The set of all characters of $G$ forms a multiplicative group and denoted by $\widehat{G}.$
\end{de}
Finite fields $ \mathbb{F}_{q^n} $ have two group structures, one is additive for $\mathbb{F}_{q^n}$ and another one is multiplicative for  $\mathbb{F}_{q^n}^*$, therefore we have two types of characters, one is an additive character for $\mathbb{F}_{q^n}$  denoted by $\psi$ and the another one is a multiplicative character for $\mathbb{F}_{q^n}^*$ denoted by $\chi.$ The additive character $\psi_0$ defined by $\psi_{0}(\beta)=e^{2\pi i \mathrm{Tr}(\beta)/p}, \ \text{for all} \ \beta \in \mathbb{F}_{q^n},$
where $\mathrm{Tr}$ is the absolute trace function from $\mathbb{F}_{q^n}$ to    $\mathbb{F}_p$, is called the {\it canonical additive character} of $\mathbb{F}_{q^n}$ and every additive character $\psi_\beta$ for $\beta \in \mathbb{F}_{q^n}$ can be expressed in terms of the canonical additive character $\psi_0$ as $\psi_\beta(\gamma)=\psi_{0}(\beta\gamma),\ \text{for all} \ \gamma \in \mathbb{F}_{q^n}.$		
\begin{de}[$\bm{e}$\textbf{-free element}]
	Let $e\mid q^n-1$ and $\beta \in \mathbb{F}_{q^n}.$ If $\beta\neq\gamma^d$ for any  $\gamma  \in \mathbb{F}_{q^n}$ and for any divisor $d$ of $e$ other than 1, then $\beta$ is called an {\it e-free element}.  Clearly, an element $\beta \in \mathbb{F}_{q^n}^*$ is primitive if and only if it is $(q^n-1)$-free.
\end{de}
Let $E$ be the set of $e$-free elements of $\mathbb{F}_{q^n}^*.$ Cohen and Huczynska \cite{CohenConsecutive} defined the characteristic function $\rho_e:\mathbb{F}_{q^n}^*\to \{0,1\}$ for $E$ by
$$\rho_{e}(\beta)= \frac{\phi(e)}{e}\sum_{d|e}\frac{\mu(d)}{\phi(d)}\sum_{\chi_{d}}\chi_{d}(\beta)$$ where $\mu$ is the M\"obius function and $\chi_d$ is the arbitrary multiplicative character of order $d$ and the internal sum runs over all such multiplicative characters.

An action of $\mathbb{F}_q[x]$ on
the additive group $\mathbb{F}_{q^n}$ is defined  by
$$g \circ \beta=\sum_{i=0}^{m}a_i\beta^{q^i}; \ \text{for} \ \beta \in \mathbb{F}_{q^n}\ \text{and} \ g(x)=\sum_{i=0}^{m}a_ix^i \in \mathbb{F}_{q}[x].$$

\begin{de}[$\bm{\mathbb{F}_{q}}$\textbf{-order of an element}]
	Let $\beta \in \mathbb{F}_{q^n}.$ Then the $\mathbb{F}_q$-order of $\beta$ is the least degree monic divisor $g$ of $x^n-1$ such that $g\circ\beta=0$.   
	
\end{de}

\begin{de}[$\bm{g}$\textbf{-free element}]
	Let $g\mid x^n-1$ and $\beta \in \mathbb{F}_{q^n}.$ If $\beta\neq h\circ \gamma$ for any  $\gamma  \in \mathbb{F}_{q^n}$ and for any divisor $h$ of $g$ other than 1, then $\beta$ is called a  $g$-free element.  Clearly, an element $\beta \in \mathbb{F}_{q^n}^*$ is normal if and only if it is $(x^n-1)$-free.
\end{de}

\begin{de}[${\mathbb{F}_{\bm{q}}}$\textbf{-order of an additive character}]
	Let $\psi \in \widehat{\mathbb{F}}_{q^n}.$ Then the $\mathbb{F}_q$-order of $\psi$ is the least degree monic divisor $g$ of $x^n-1$ such that $\psi \circ g$ is a trivial character, where $\psi \circ g(\beta)=\psi(g \circ \beta).$  
\end{de}

Let $F$ denotes the set of $g$-free elements of $\mathbb{F}_{q^n}.$ Then the characteristic function $\kappa_{g}: \mathbb{F}_{q^n}\to \{0,1\}$ for $F$ is given by
$$\kappa_{g}(\beta)= \dfrac{\Phi_q(g)}{q^{\mathrm{deg}(g)}}\sum_{h|g}\dfrac{\mu'(h)}{\Phi_{q}(h)}\sum_{\psi_{h}}\psi_{h}(\beta),$$
where $\Phi_q(g)=\Big|\Big(\frac{\mathbb{F}_q[x]}{<g>}\Big)^*\Big|$, $\mu'$ is the analogue of the M\"obius function, which is defined as $\mu'(g)=(-1)^s$, if $g$ is a product of $s$ distinct monic irreducible polynomials, otherwise 0, and the internal sum runs over the additive characters $\psi_{h}$ of the $\mathbb{F}_{q}$-order $h$.

%Lemma 2.1-------------------------
\begin{lemma}\label{l2.1}
	\cite[Theorem 5.5]{Fu} Let $f \in \mathbb{F}_{q^n}(x)$ be a rational function. Write $f=\prod_{j=1}^{k}f_j^{n_j},$ where $f_j\in \mathbb{F}_{q^n}[x]$ are irreducible polynomials and $n_j$ are non-zero integers. Let $\chi$ be a multiplicative character of order $d$ of $\mathbb{F}_{q^n}^*.$ Suppose that the rational function $f$ is not of the form $h^{d}$ for any $h\in \mathbb{F}(x),$ where $\mathbb{F}$ is the algebraic closure of $\mathbb{F}_{q^n}.$ Then we have  
	$$\Big |\sum_{\substack{\alpha \in \mathbb{F}_{q^n}\\f(\alpha)\neq 0,\infty}}\chi(f(\alpha))\Big|\leq \Big(\sum_{j=1}^{k}\mathrm{deg}(f_j)-1\Big)q^{n/2}.$$
\end{lemma}
%lemma2.2---------------------------------
\begin{lemma}\label{l2.2}
	\cite[Theorem 5.6]{Fu} Let $f,g \in \mathbb{F}_{q^n}(x)$ be rational functions. Write $f=\prod_{j=1}^{k}f_j^{n_j}$, where $f_j\in \mathbb{F}_{q^n}[x]$ are irreducible polynomials and $n_j$ are non-zero integers. Let $D_1=\sum_{j=1}^{k}\mathrm{deg}(f_j)$, $D_2=max\{\mathrm{deg}(g),0\}$, $D_3$ be the degree of denominator of $g$ and $D_4$ be the sum of degrees of those irreducible polynomials dividing the denominators of $g$, but distinct from $f_j \ ;\ 1\leq j\leq k.$ Let $\chi$ be the multiplicative character of $\mathbb{F}_{q^n}^*$ and $\psi$ be the non-trivial additive character of $\mathbb{F}_{q^n}.$ Suppose that $g\neq h^{q^n}-h$ in $\mathbb{F}(x),$ where $\mathbb{F}$ is the algebraic closure of $\mathbb{F}_{q^n}.$ Then we have 
	$$\Big |\sum_{\substack{\alpha \in \mathbb{F}_{q^n},f(\alpha)\neq 0,\infty \\ g(\alpha)\neq \infty}} \chi(f(\alpha))\psi(g(\alpha))\Big|\leq (D_1+D_2+D_3+D_4-1)q^{n/2}.$$ 
\end{lemma}

Let $d\mid q^n-1$ and $g\mid x^n-1.$ We denote the number of square free divisors of $d$ by $W(d)$ and of $g$  by $W(g).$
%lemma2.3-----
\begin{lemma}\label{l2.3}
	\cite[Lemma 2.9]{Lenstra} Let $q$ be a prime power and $n$ be a positive integer. Then we have $W(x^n-1)\leq2^{\frac{1}{2} (n+\mathrm{gcd}(n,q-1))}$. In particular, $W(x^n-1)\leq 2^n$ and $W(x^n-1)=2^n$ if and only if $n\mid(q-1)$. Furthermore, $W(x^n-1) \leq 2^{3n/4}$ if $n\nmid(q-1),$ since in this case, $\mathrm{gcd}(n,q-1)\leq n/2.$
\end{lemma}
%lemma2.4.----
\begin{lemma}\label{l2.4}{\cite[Lemma 3.7]{CoHuc}}
	For any $\eta \in \mathbb{N}$ and $\nu$ a positive real number, $W(\eta)\leq \C\cdot \eta^{1/\nu}$, where $\C=\frac{2^s}{(p_1p_2\ldots p_s)^{1/\nu}}$ and $p_1,p_2,\ldots p_s$ are the primes $\leq 2^\nu $ that divide $\eta.$
\end{lemma}

If $\mathbb{F}_{q^n}$ is the extension field of $\mathbb{F}_{q}$ of degree $n$ such that gcd$(q,n)=1$ and $d$ is the least positive integer such that $q^d\equiv1$ (mod $n$), then from \cite[Theorem 2.45, 2.47]{Nieder}, $x^n-1$ splits into the product of irreducible polynomials of degree less than or equal to $d$ in $\mathbb{F}_{q}[x].$ Let $N_0$ denotes the number of distinct irreducible factors of $x^n-1$ over $\mathbb{F}_{q}$ of degree less than $d$ and  $\rho(q,n)$ denotes the ratio $\frac{N_0}{n}.$ If $m=n\cdot p^i$, where $p$ is a prime such that $q=p^k$ and coprime to $n$, then the number of irreducible factors of $x^m-1$ over $\mathbb{F}_{q}$ is the same as that of $x^n-1$ over $\mathbb{F}_{q}$, and hence, $m\rho(q,m)=n\rho(q,n).$ In \cite{SDEven}, Cohen provided the following bounds on the ratio $\rho(q,n)$.

%Lemma2.5------------
\begin{lemma}\label{l2.5}{\cite[Lemma 7.1]{SDEven}}
	Let $q=2^k$ and $n$ be an odd integer. Then the following hold.
	\begin{enumerate}[$(i)$]
		\item If $q=2$ then  $\rho(2,3)=1/3; \ \rho(2,5)=1/5; \ \rho(2,9)=2/9; \ \rho(2,21)=4/21;$ otherwise $\rho(2,n)\leq1/6.$
		\item If $q=4$ then $\rho(4,9)= 1/3; \ \rho(4,5)=11/45;$ otherwise $\rho(4,n)\leq1/5.$
		\item If $q=8$ then $\rho(8,3)= \rho(8,21)=1/3;$ otherwise $\rho(8,n)\leq1/5.$
		\item  If $q\geq16$ then $\rho(q,n)\leq 1/3.$
	\end{enumerate}
\end{lemma}

%Section3----------------------------------------------------------------------------------------------	
\section{Primitive pairs in $\mathbb{F}_{2^{k}}$}\label{Sec3}

Let $f=\frac{f_1}{f_2}\in \mathbb{F}_{2^k}(x)$, where deg$f_1=$ deg$f_2=2$ and $f_2\nmid f_1.$
For each such rational function $f,$ R. K. Sharma et al. \cite{Awasthi} proved that there always exists a primitive pair $(\alpha,f(\alpha))$ in $\mathbb{F}_{2^k}$ unless $k=1,2,4,6,8,9,10,12.$ Moreover, they found counter examples for $k=1,2,4$ and conjectured that the existence of such a primitive pair is guaranteed for $k=6,8,9,10,12.$ Recently, this conjecture has been proved for the case $k=9$ by Carvalho et al. \cite{CarNeu}. 

Now, we show that there exists a rational function $f=\frac{f_1}{f_2}\in \mathbb{F}_{2^3}(x)$ such that $(\alpha,f(\alpha))$ is not a primitive pair for any primitive element $\alpha$ in $\mathbb{F}_{2^3}^*.$ Note that $2^3-1=7$ is a prime, hence, every non-trivial element of $\mathbb{F}_{2^3}^*$ is a primitive element. Clearly, for any primitive element $\alpha\in \mathbb{F}_{2^3}^*$, $f(\alpha)$ cannot be a primitive element if $f_1(\alpha)=0$ or $f_2(\alpha)=0$ or $f_1(\alpha)=f_2(\alpha).$ These equations can be satisfied by at most 6 distinct primitive elements. Since $\mathbb{F}_{2^3}^*$ contains exactly 6 primitive elements, there may exists a rational function $f$ for which $f(\alpha)$ is not a primitive element for any primitive element $\alpha$. Indeed, such a rational function exists and we give the following example. Hence, the case $k=3$ is also a genuine exception.

\begin{eg}
	Let  $f=\dfrac{f_1}{f_2}=\dfrac{\beta x^2+x+\beta}{x^2+(\beta+1)x+1},$ where $\beta$ is a root of primitive polynomial \ $x^3+x+1$  over $\mathbb{F}_2.$ Then $f(\alpha)$ is not a primitive element for any primitive element $\alpha$ of \ $\mathbb{F}_{2^3}$, because $f_1(\beta^3)=f_1(\beta^4)=0,\ f_2(\beta^2)=f_2(\beta^5)=0$ and   $f_1(\beta^i)=f_2(\beta^i)$ for $i=1,6.$ 
\end{eg}

%Section4 Sufficient condition----------------

\section{A sufficient condition }\label{Sec4}
In this section, we provide a sufficient condition for the existence of a primitive normal pair $(\alpha, f(\alpha))$ in $\mathbb{F}_{q^n}$ over $\mathbb{F}_q$ for every rational function $f$ belongs to the set $R_{q^n}(m_1,m_2)$, where:
\begin{de}\label{def4.1}
	 The set
	 $R_{q^n}(m_1,m_2)$ is the collection of rational functions $f\in \mathbb{F}_{q^n}(x)$ with the simplest form $\frac{f_1}{f_2},$ where $m_1,m_2$ are the degrees of $f_1,f_2$ respectively, such that 
	\begin{enumerate}[$(i)$]
		\item $f\neq cx^jh^d$ for  any $ c\in \mathbb{F}_{q^n}^*$, any $ j\in \mathbb{Z}$, any $h\in \mathbb{F}_{q^n}(x)$,  and any prime divisor $ d$ of $q^n-1,$
		\item there exists at least one monic irreducible factor $g$ of $f_2$ in $\mathbb{F}_{q^n}[x]$ with multiplicity $r$ such that $q^n\nmid r.$
\end{enumerate} 
\end{de}
Note that, (ii) part of the above definition implies that $m_2\geq1.$ For the case $m_2=0,$ i.e. for polynomials $f\in \mathbb{F}_{q^n}[x]$, the existence of a primitive normal element $\alpha \in \mathbb{F}_{q^n}$ over $\mathbb{F}_q$ such that $f(\alpha)$ is also a primitive normal element in $\mathbb{F}_{q^n}$ over $\mathbb{F}_q$ has been studied in \cite{rani}.

Now, for a given $f \in R_{q^n}(m_1,m_2)$, $l_1,l_2 \mid q^n-1$ and  $h_1,h_2\mid x^n-1,$ we denote $N_f(l_1,l_2,h_1,h_2)$ as the number of $\alpha \in \mathbb{F}_{q^n}$ such that $\alpha$ is $l_1$-free as well as $h_1$-free and $f(\alpha)$ is $l_2$-free as well as $h_2$-free. We also denote $Q_p(m_1,m_2)$ as the collection of pairs $(q,n)$ for which  there always exists a primitive normal pair $(\alpha,f(\alpha))$ in $\mathbb{F}_{q^n}$ over $\mathbb{F}_q$ for every $f\in R_{q^n}(m_1,m_2).$
Now, we give the following sufficient condition.

%Theorem 4.1-----------
\begin{theorem}\label{T4.1}
	Let $q=p^k$, for some prime $p$; $k\in \mathbb{N}$, $n$ be a positive integer and  $f \in R_{q^n}(m_1,m_2)$ be a rational function. Then $N_f(l_1,l_2,h_1,h_2)>0$ if
	$$q^{n/2}> \left\{\begin{array}{lll}
	(2m_1+m_2)W(l_1)W(l_2)W(h_1)W(h_2)&;& \text{if}\ m_1>m_2\\
	(m_1+2m_2+1)W(l_1)W(l_2)W(h_1)W(h_2)&;& \text{if}\ m_1\leq m_2
	\end{array}\right..$$
	
\end{theorem} 
\begin{proof}
	Let $S$ be the set containing 0, and zeroes and poles of $f$, then by definition
	\begin{equation*}
	N_f(l_1,l_2,h_1,h_2)=\sum_{\alpha \in \mathbb{F}_{q^n} \setminus S}\rho_{l_1}(\alpha)\rho_{l_2}(f(\alpha))\kappa_{h_1}(\alpha)\kappa_{h_2}(f(\alpha))
	\end{equation*}
	\begin{equation}\label{Eq1}
	N_f(l_1,l_2,h_1,h_2)=\ {\theta} {\Theta}\sum_{\substack{d_1|l_1,d_2|l_2\\ g_1|h_1,g_2|h_2}} \dfrac{\mu(d_1)\mu(d_2)\mu'(g_1)\mu'(g_2)}{\phi(d_1)\phi(d_2)\Phi(g_1)\Phi(g_2)} \sum_{\substack{\chi_{d_1},\chi_{d_2}\\\psi_{g_1},\psi_{g_2}}}{\chi}_{d_1,d_2,g_1,g_2}
	\end{equation}
	where ${\theta}=\frac{\phi(l_1)\phi(l_2)}{l_1l_2}$, ${\Theta}=\frac{\Phi(h_1)\Phi(h_2)}{q^{\mathrm{deg}(h_1)}q^{\mathrm{deg}(h_2)}}$ and $${\chi}_{d_1,d_2,g_1,g_2}=	\sum\limits_{\alpha \in \mathbb{F}_{q^n}\setminus S}\chi_{d_1}(\alpha)\chi_{d_2}(f(\alpha))\psi_{g_1}(\alpha)\psi_{g_2}(f(\alpha)).$$
	
	Let $d_1,d_2 \mid q^n-1$, then $\chi_{d_1}=\chi_{q^n-1}^{n_1}$ and  $\chi_{d_2}=\chi_{q^n-1}^{n_2}$ for some multiplicative character $\chi_{q^n-1}$ of order $q^n-1,$ where $n_1\in \{0,1,\dots,q^n-2\}$ and  $n_2=(q^n-1)/d_2.$ Moreover, let $g_1,g_2 \mid x^n-1,$ then there exist $y_1,y_2 \in \mathbb{F}_{q^n}$ such that $\psi_{g_i}(\alpha)=\psi_{0}(y_i\alpha)$; $i=1,2$, where $\psi_0$ is the canonical additive character of $\mathbb{F}_{q^n}.$ Hence,
	\begin{align*}
	{\chi}_{d_1,d_2,g_1,g_2}&=	\sum_{\alpha \in \mathbb{F}_{q^n}\setminus S}\chi_{q^n-1}(\alpha^{n_1}(f(\alpha)^{n_2})\psi_{0}(y_1\alpha+y_2f(\alpha))\\
	&=	\sum_{\alpha \in \mathbb{F}_{q^n}\setminus S}\chi_{q^n-1}(F(\alpha))\psi_{0}(G(\alpha)),
	\end{align*}
	where $F(x)=x^{n_1}f^{n_2}$ and $G(x)=y_1x+y_2f$ are the polynomials in $\mathbb{F}_{q^n}(x)$ and $\mathbb{F}(x)$ respectively. \\
	
	{\it Claim :} If $F(x)= h^{q^n-1}$ for some $h \in \mathbb{F}(x)$, then $d_1=d_2=1$, and if $G(x)= h^{q^n}-h$ for some $h \in \mathbb{F}(x)$, then $g_1=g_2=1,$ where $\mathbb{F}$ is the algebraic closure of $\mathbb{F}_{q^n}.$
	
	{\it Proof of the claim :}
	Firstly, let $F(x)=h^{q^n-1}$ for some $h\in \mathbb{F}(x).$ Write $f=c_1x^jf_1/f_2$ and $h=c_2x^kh_1/h_2$, where $c_1\in \mathbb{F}_{q^n}^*$, $c_2\in \mathbb{F}^*$, $j,k\in \mathbb{Z}$, $f_1,f_2$ and $h_1,h_2$ are monic polynomials besides $x$ in $\mathbb{F}_{q^n}[x]$ and $\mathbb{F}[x]$ respectively, such that $\mathrm{gcd}(f_1,f_2)=1$ and $\mathrm{gcd}(h_1,h_2)=1$. Then 
	\begin{equation}\label{Eq2}
	c_1^{n_2}x^{n_1+jn_2}f_1^{n_2}h_2^{q^n-1}=c_2^{q^n-1}x^{k(q^n-1)}h_1^{q^n-1}f_2^{n_2}.
	\end{equation}
	Now, let $f_1=\prod_{i=1}^{t}f_{1i}^{s_i}$ be the decomposition of $f_1$ into monic irreducible polynomials over $\mathbb{F}_{q^n}$, where $s_i$'s are positive integers. Since $\mathbb{F}$ is the algebraic closure of $\mathbb{F}_{q^n}$, $f_{1i}$ completely reduces into distinct linear factors over $\mathbb{F}.$ From Equation \eqref{Eq2}, it is clear that, each such linear factor of $f_{1i}$ divides $h_1$ over $\mathbb{F}$, since gcd$(f_1,f_2)=1$ and $x\nmid f_1$. Therefore, $f_{1i}\mid h_1$ over $\mathbb{F}$ for all $i=1,2,\ldots,t.$ By comparing the degrees of the factors $f_{1i}^{s_in_2}$ on both sides of Equation \eqref{Eq2} over $\mathbb{F}$, we get $s_in_2=k_i(q^n-1)$ i.e. $s_i=k_id_2$ for some positive integer $k_i$; $i=1,2,\ldots,t.$ Hence, $f_1={f_1'}^{d_2}$ for some ${f_1'}\in \mathbb{F}_{q^n}[x].$ Similarly, $f_2={f_2'}^{d_2}$ for some ${f_2'}\in \mathbb{F}_{q^n}[x].$ Thus, $f=c_1x^j{f'}^{d_2}$, where ${f'}={f_1'}/f_{2}'.$ Since $f\in R_{q^n}(m_1,m_2)$, $d_2=1$, and hence, $n_2=q^n-1.$ By comparing the degree of $x$ on both sides of Equation \eqref{Eq2} over $\mathbb{F}$, we get $n_1+jn_2=k(q^n-1)$. This implies $n_1=(k-j)(q^n-1)$. Which further implies $n_1=0$, since $0\leq n_1<q^n-1.$ Hence, if $F(x)=h^{q^n-1}$ for some $h\in \mathbb{F}(x),$ then $n_1=n_2=0$ that means $d_1=d_2=1.$
	
	Secondly, let $G(x)=h^{q^n}-h$ for some $h\in \mathbb{F}(x)$ and at least one of $y_1$ and $y_2$ is non-zero. Write $f=\frac{f_1}{f_2}$ and $h=\frac{h_1}{h_2}$, where $f_1,f_2$ and $h_1,h_2$ are polynomials in $\mathbb{F}_{q^n}[x]$ and $\mathbb{F}[x]$ respectively, such that $\mathrm{gcd}(f_1,f_2)=1$ and $\mathrm{gcd}(h_1,h_2)=1$, then we have 
	\begin{equation}\label{Eq3}
	h_2^{q^n}(y_1xf_2+y_2f_1)=f_2(h_1^{q^n}-h_1h_2^{q^n-1}).
	\end{equation}
	For $y_2\neq0$, the conditions gcd$(f_1,f_2)=1$ and gcd$(h_1,h_2)=1$ imply that $f_2=h_2^{q^n}$ in Equation \eqref{Eq3}.
    From the definition of $R_{q^n}(m_1,m_2)$, there exists a monic irreducible factor $g\in \mathbb{F}_{q^n}[x]$ of $f_2$ with multiplicity $r$ such that $q^n\nmid r.$ So, by comparing the multiplicity of this factor on both sides of the equation $f_2=h_2^{q^n},$ we get $r=sq^n$, where $s$ is the multiplicity of $g$ in the factorization of $h_2$ over $\mathbb{F}.$ Clearly, it is not possible as $q^n \nmid r.$ Therefore, $y_2=0.$ Now, Equation \eqref{Eq3} becomes $y_1xh_2^{q^n}=h_1^{q^n}-h_1h_2^{q^n-1}.$ Moreover, $h_2$ becomes a non-zero constant in $\mathbb{F}$, since gcd$(h_1,h_2)=1$. For $y_1\neq 0$, the equation $y_1xh_2^{q^n}=h_1^{q^n}-h_1h_2^{q^n-1}$ implies that $1=q^n\mathrm{deg}(h_1),$ which is not possible. Therefore, $y_1=0$. Hence, if $G(x)=h^{q^n}-h$ for some $h\in \mathbb{F}(x)$, then $y_1=y_2=0$ that means $g_1=g_2=1$.
	
	Now, let $G(x)\neq h^{q^n}-h$ for any $h\in \mathbb{F}(x)$, then by Lemma \ref{l2.2}, we get
	$$|{\chi}_{d_1,d_2,g_1,g_2}|\leq Mq^{n/2},\  
	\text{where}\  M=\left\{\begin{array}{lll}
	2m_1+m_2&;& \text{if}\ m_1>m_2\\
	m_1+2m_2+1&;& \text{if}\ m_1\leq m_2
	\end{array}\right..$$ 
	
	If $G(x)=h^{q^n}-h$ for some $h\in \mathbb{F}(x),$ then from our claim, we have $g_1=g_2=1.$  Additionally, if $F(x)\neq h^{q^n-1}$ for any $h \in \mathbb{F}(x)$, then Lemma \ref{l2.1} gives
	$$|{\chi}_{d_1,d_2,g_1,g_2}|\leq (m_1+m_2)q^{n/2}.$$ 	
	So, the only remaining case is when $G(x)=h^{q^n}-h$ for some $h\in \mathbb{F}(x)$ and  $F(x)= h^{q^n-1}$ for some $h \in \mathbb{F}(x)$ i.e. the trivial case $d_1=d_2=1,\ g_1=g_2=1.$ In this case, we have  $$|{\chi}_{1,1,1,1}|=q^n-|S|\geq q^n-(m_1+m_2+1).$$ Therefore, from Equation \eqref{Eq1}, we get
	\begin{equation*}
	N_f(l_1,l_2,h_1,h_2)>{\theta\Theta}\{q^n-(m_1+m_2+1)-Mq^{n/2}(W(l_1)W(l_2)W(h_1)W(h_2)-1)\}.
	\end{equation*}
	Hence, $N_f(l_1,l_2,h_1,h_2)>0$ if $q^{n/2}>MW(l_1)W(l_2)W(h_1)W(h_2),$ where 
	\begin{equation*}
	M=\left\{\begin{array}{lll}
	2m_1+m_2&;& \text{if}\ m_1>m_2\\
	m_1+2m_2+1&;& \text{if}\ m_1\leq m_2
	\end{array}\right.. \qedhere
	\end{equation*}
\end{proof}

The above inequality provides a sufficient condition for the existence of  primitive normal pairs if we take $l_1=l_2=q^n-1$ and $h_1=h_2=x^n-1,$ i.e. $N_f:=N_f(q^n-1,q^n-1,x^n-1,x^n-1)>0$ if  
\begin{equation}\label{Eq4}
q^{n/2}>MW(q^n-1)^2W(x^n-1)^2.
\end{equation}

\begin{rk}\label{rk1}
	From {\upshape Lemma \ref{l2.4}}, we have $W(q^n-1)<\C(q^n-1)^{1/\nu}$ for any positive real number $\nu.$ Moreover, the number of square free divisors of $x^n-1$ will be less than or equal to $2^n$ i.e. $W(x^n-1)\leq 2^n.$ Hence, Inequality \eqref{Eq4} holds if 
	\begin{equation}\label{Eq5}
	q^{n/2} > M \C^2q^{2n/\nu} 2 ^{2n}.
	\end{equation}
	In particular, if $f\in R_{q^n}(2,2)$, then $M=7.$ Furthermore, for $\nu=6.5,$ {\upshape Lemma \ref{l2.4}} gives  $\C\leq86.32.$ Hence, for these values of $M$ and $\C,$ Inequality \eqref{Eq5} holds true for $q>2^{37.56}$ and $n\geq 3$ i.e. $(q,n)\in Q_p(2,2)$  for $q>2^{37.56}$ and $n\geq 3$. 	
\end{rk}

\begin{rk}
	Observe that, if $g\in R_{q^n}(m_1,m_2),$ then $g^{q^i}$ may not belong to $R_{q^n}(m_1 q^i,m_2 q^i).$ For example $g=\frac{x}{x+1}\in R_{2}(1,1)$ but $g^2\notin R_{2}(2,2)$ as the $(ii)$ part in {\upshape Definition \ref{def4.1}} of $R_{q^n}(m_1,m_2)$ is not satisfied. But, we know that if $\alpha$ is a primitive normal element in $\mathbb{F}_{q^n}$ over $\mathbb{F}_{q}$, then $\alpha^{q^i}$ is also a primitive normal element in $\mathbb{F}_{q^n}$ over $\mathbb{F}_{q}$ for every non-negative integer $i.$  Therefore, if a rational function $f$ is of form $g^{q^i}$ for some rational function $g\in R_{q^n}(m_1,m_2)$, then the existence of a primitive normal pair $(\alpha,g(\alpha))$ implies the existence of a primitive normal pair $(\alpha,f(\alpha)).$ 
\end{rk}

We give a modified form of Inequality \eqref{Eq4} in Theorem \ref{T4.2}  by using the following lemma. 
%Lemma4.1----------------
\begin{lemma}\cite[Proposition 5.2]{KapeNBT}\label{l4.1} {\upshape (Sieving inequality)}
	Let $d\mid q^n-1$ and $p_1,p_2,\ldots,p_s$ be the remaining distinct primes dividing $q^n-1.$ Furthermore, let $g\mid x^n-1$ and $g_1,g_2,\ldots,g_t$ be the remaining distinct irreducible polynomials dividing $x^n-1.$ Then
	\begin{align*}
	N_f\geq& \sum_{i=1}^{s}N_f(p_id,d,g,g)+\sum_{i=1}^{s}N_f(d,p_id,g,g)+\sum_{i=1}^{t}N_f(d,d,g_ig,g) \ +
	\\&\sum_{i=1}^{t}N_f(d,d,g,g_ig)-(2s+2t-1)N_f(d,d,g,g)
	\end{align*}
\end{lemma}
%Theorem4.2-------------------
\begin{theorem}\label{T4.2}
	Let $d\mid q^n-1$ and $p_1,p_2,\ldots,p_s$ be the remaining distinct primes dividing $q^n-1.$ Furthermore, let $g\mid x^n-1$ and $g_1,g_2,\ldots,g_t$ be the remaining distinct irreducible polynomials dividing $x^n-1.$ Define $$\mathcal{D}:=1-2\sum\limits_{i=1}^{s}\frac{1}{p_i}-2\sum\limits_{i=1}^{t}\frac{1}{q^{\mathrm{deg}(g_i)}} \ \  and \ \ \mathcal{S}:=\frac{2s+2t-1}{\mathcal{D}}+2.$$ Suppose $\mathcal{D}>0,$ then $N_f>0$, if $q^{n/2}>MW(d)^2W(g)^2\mathcal{S},$
	where $M$ is defined as in \upshape{Theorem \ref{T4.1}}. 
\end{theorem}

\begin{proof}
	Proof is similar to that of \cite[Proposition 5.3]{KapeNBT}, hence omitted.	
\end{proof}
%Section5-------------------------------------
\section{Pairs $(q,n)\in Q_2(2,2)$}\label{Sec5}

From Remark \ref{rk1}, we have a lower bound on the order of finite fields $\mathbb{F}_{q^n}$ for which ($q,n$) always belongs to $Q_p(2,2)$ for any prime $p$.  In this section, we confine our study to $q=2^k$ and the set $R_{q^n}(2,2).$ We present a complete list of pairs $(q, n)$ for which the existence of the primitive normal pair $(\alpha,f(\alpha))$ is not guaranteed. Clearly, for $f\in R_{q^n}(2,2),$ Inequality \eqref{Eq5} becomes
\begin{equation}\label{Eq6}
q^{n/2} > 7 \C^2 q^{2n/\nu} 2 ^{2n}.
\end{equation}
The above inequality is equivalent to $	(\frac{\nu-4}{2\nu})\hspace{.5mm} {\mathrm{log}}\hspace{.5mm}2^k-2\hspace{.5mm}{\mathrm{log}}\hspace{.5mm}2 > \frac{{\mathrm{log}}(7 \C^2)}{n}$, which is valid if $k>\frac{4\nu}{\nu-4}$ i.e. $k\geq5$. The computations, wherever needed in this section, are done using SageMath \cite{sagemath}. We prove the following lemmas.
\begin{lemma}\label{l5.1}
	Let $ q=2^k$ and $n\geq3$ be a positive integer. If $k\geq 6,$ then $ (q,n)\in Q_2(2,2)$ unless $q=64$ and $n=3$. 
\end{lemma}
\begin{proof}
	For $\nu=6.5$, Lemma \ref{l2.4} gives $\C\leq 48.016$ and Inequality \eqref{Eq6} is satisfied if $k\geq35$ and $n\geq3.$ For $6\leq k\leq34$, Table \ref{Table1} provides the values of $n$ for suitable choices of $\nu$ such that Inequality \eqref{Eq6} is true.
For the remaining values of $k$ and $n$, we test the inequality $q^{n/2}>7W(q^n-1)^2W(x^n-1)^2$ and get $(q,n)\in Q_2(2,2)$ unless
\begin{enumerate}[(i)]
	\item $k=6$ and $n=3,4,5,6,7,9,10,21$ 
	\item $k=7$ and $n=3,4$
	\item $k=8$ and $n=3,5,6,9,15$
	\item $k=9$ and $n=4$
	\item $k=10$ and $n=3,6$
	\item $k=12$ and $n=3,5$
	\item $k=16$ and $n=3$ 
	\item $k=20$ and $n=3$
\end{enumerate}  
For these values of $k$ and $n$, we choose the values of $d$ and $g$ listed in Table \ref{Table2} such that the inequality $q^{n/2}>7 W(d)^2 W(g)^2 \mathcal{S}$ holds and get that $(q,n)\in Q_2(2,2)$ unless $q=64$ and $n=3$.
\end{proof}
{\small
\begin{longtable}[h]{|ccc|ccc|ccc|}
	\caption{Values of $n$ for $6\leq k\leq 35$.\label{Table1}}\\
	\hline
	$k(=)$&$n(\geq)$&$\nu(=)$&$k(=)$&$n(\geq)$&$\nu(=)$&$k(=)$&$n(\geq)$&$\nu(=)$\\
	\hline
	\endfirsthead
	\multicolumn{7}{l}{Continuation of Table \ref{Table1}}\\
	\hline
	$k(=)$&$n(\geq)$&$\nu(=)$&$k(=)$&$n(\geq)$&$\nu(=)$&$k(=)$&$n(\geq)$&$\nu(=)$\\
	\hline
	\endhead
	
	\hline
	\endfoot
	
	\hline
	\endfoot
	
	6&3101&13.7&16 & 12 & 7&26 & 5 & 7\\
	7&432&11&17 & 11 & 7&27 & 5 & 7\\
	8&152&9.8&18 & 10 & 7&28 & 5 & 7\\
	9&78&9&19 & 9 & 7&29 & 4 & 6.5\\
	10&49&8.5&20 & 8 & 7&30 & 4 & 7\\
	11&34&8&21 & 7 & 7&31 & 4 & 7\\
	12&26&8&22 & 7 & 7&32 & 4 & 7\\
	13&20&7.8&23 & 6 & 7&33 & 4 & 7\\
	14&17&7.5&24 & 6 & 7&34 & 4 & 7\\
	15&14&7&25 & 5 & 6.5&35 & 3 & 6.5\\
	
	\hline
\end{longtable}
}
For the remaining cases $1 \leq k\leq5$, we write $n=m\hspace{.5mm} 2^i$; $i\geq0$, where gcd$(m,2)=1$ and divide our discussion in the following two cases:
\begin{itemize}
	\item $m\mid q^2-1$
	\item $m\nmid q^2-1$
\end{itemize}
We first assume that $m\mid q^2-1$ and prove the following lemma.
%Lemma5.2---------------------------	
\begin{lemma}\label{l5.2}
	Let $q=2^k$, where $1\leq k\leq5$ , $n\geq 3
	$ and $m\mid q^2-1$. Then $(q,n)\in Q_2(2,2)$ unless 
	\begin{enumerate}[$(i)$]
		\item $q=2$ and $n=3, 4, 6, 8, 12, 16, 24.$
		\item $q=4$ and $n=3, 4, 5, 6, 8, 10, 12, 15.$
		\item $q=8$ and $n=3, 4, 6, 7, 8, 14.$
		\item $q=16$ and $n=3, 4, 5, 6, 15.$	
		
	\end{enumerate}
\end{lemma}

\begin{proof}
	We rewrite Inequality \eqref{Eq6} as follows
	\begin{equation}\label{Eq7}
	q^{m2^i/2}>7 \C^2 q^{2m 2^i/\nu} 2^{2m}.
	\end{equation}
	\noindent
	{${\textbf{Case} \ \bm {k=1:}}$} Since $m\mid q^2-1$, $m=1,3$ and Inequality \eqref{Eq7} holds true for $i\geq7$, when $m=1$ and for $i\geq5$, when $m=3$. Hence, $(2,n)\in Q_2(2,2)$ unless $n=$ 3, 4, 6, 8, 12, 16, 24, 32, 48, 64. For these remaining values, we directly verify the inequality $q^{n/2}>7 W(q^n-1)^2 W(x^n-1)^2$ and get that $(2,n)\in Q_2(2,2)$ unless $n=$ 3, 4, 6, 8, 12, 16, 24, 48. For these values of $n$, we choose the values of $d$ and $g$ listed in Table \ref{Table2} such that the inequality $q^{n/2}>7 W(d)^2 W(g)^2 \mathcal{S}$ holds true and get that $(2,n)\in Q_2(2,2)$ unless $n=$ 3, 4, 6, 8, 12, 16, 24.\\
	
	\noindent
	{${\textbf{Case} \ \bm {k=2:}}$} In this case $m =$ 1, 3, 5, 15 and Inequality \eqref{Eq7} holds for $i\geq6$, when $m=1$, for $i\geq5$, when $m=3$, for $i\geq4$, when $m=5$ and for $i\geq3$, when $m=15$. Hence, $(4,n)\in Q_2(2,2)$ unless $n=$ 3, 4, 5, 6, 8, 10, 12, 15, 16, 20, 24, 30, 32, 40, 48, 60. For these remaining values, we directly verify the inequality $q^{n/2}>7 W(q^n-1)^2 W(x^n-1)^2$ and get that $(4,n)\in Q_2(2,2)$ unless $n=$ 3, 4, 5, 6, 8, 10, 12, 15, 20, 24, 30. For these values of $n$, we choose the values of $d$ and $g$ listed in Table \ref{Table2} such that the inequality $q^{n/2}>7 W(d)^2 W(g)^2\mathcal{S}$ holds and get that $(4,n)\in Q_2(2,2)$ unless $n=$ 3, 4, 5, 6, 8, 10, 12, 15.\\
	
	\noindent
	{${\textbf{Case} \ \bm {k=3:}}$} In this case $m=$ 1, 3, 7, 9, 21, 63 and Inequality \eqref{Eq7} holds for $i\geq5$, when $m=1$, for $i\geq4$, when $m=3$, for $i\geq3$, when $m=$ 7, 9 and for $i\geq2$, when $m=$ 21, 63. Hence, $(8,n)\in Q_2(2,2)$ unless $n=$ 3, 4, 6, 7, 8, 9, 12, 14, 16, 18, 21, 24, 28, 36, 42, 63, 126. For these remaining values, we directly verify the inequality $q^{n/2}>7 W(q^n-1)^2 W(x^n-1)^2$ and get that $(8,n)\in Q_2(2,2)$ unless $n=$ 3, 4, 6,7, 8, 9, 12,14, 21. For these values of $n$, we choose the values of $d$ and $g$ listed in Table \ref{Table2} such that the inequality $q^{n/2}>7W(d)^2W(g)^2\mathcal{S}$ holds and get that $(8,n)\in Q_2(2,2)$ unless $n=$ 3, 4, 6, 7, 8, 14.\\
	
	\noindent
	{${\textbf{Case} \ \bm {k=4:}}$} In this case $m=$ 1, 3, 5, 15, 17, 51, 85, 255 and Inequality \eqref{Eq7} holds for $i\geq5$, when $m=1$, for $i\geq4$, when $m=3$, for $i\geq3$, when $m=5$, for $i\geq2$, when $m=$ 15, 17 and $i\geq 1$, when $m=$ 51, 85, 255. Hence, $(16,n)\in Q_2(2,2)$ unless $n=$ 3, 4, 5, 6, 8, 10, 12, 15, 16, 17, 20, 24, 30, 34, 51, 85, 255. For these remaining values, we directly verify the inequality $q^{n/2}>7W(q^n-1)^2W(x^n-1)^2$ and get that $(16,n)\in Q_2(2,2)$ unless $n=$ 3, 4, 5, 6, 10, 12, 15, 17, 30. For these values of $n$, we choose the values of $d$ and $g$ listed in Table \ref{Table2} such that the inequality $q^{n/2}>7W(d)^2W(g)^2\mathcal{S}$ holds and get that $(16,n)\in Q_2(2,2)$ unless $n=$ 3, 4, 5, 6, 15.\\
	
	\noindent
	{${\textbf{Case} \ \bm {k=5:}}$} In this case $m=$ 1, 3, 11, 31, 33, 93, 341, 1023 and Inequality \eqref{Eq7} holds for $i\geq5$, when $m=1$, for $i\geq3$, when $m=3$, for $i\geq2$, when $m=$ 11, 31, for $i\geq 1$, when $m=$ 33, 93 and for $i\geq0$, when $m=$ 341, 1023. Hence, $(32,n)\in Q_2(2,2)$ unless $n=$ 3, 4, 6, 8, 11, 12, 16, 22, 31, 33, 62, 93. For these remaining values, we directly verify the inequality $q^{n/2}>7W(q^n-1)^2W(x^n-1)^2$ and get that $(32,n)\in Q_2(2,2)$ unless $n=$ 3, 4, 6, 31. For these values of $n$, we choose the values of $d$ and $g$ listed in Table \ref{Table2} such that the inequality $q^{n/2}>7W(d)^2W(g)^2\mathcal{S}$ holds and get that $(32,n)\in Q_2(2,2)$ unless $n=$ 4.
\end{proof}
Secondly, we assume that $m\nmid q^2-1. $ For the further discussion, first we prove the following result inspired from Cohen \cite[Lemma 7.2]{SDEven}, which provides a bound on $\mathcal{S}$ defined in Theorem \ref{T4.2}.
%Lemma5.3-----------
\begin{lemma}\label{l5.3}
	Assume that $q=p^k;\ k\in\mathbb{N}$ and $n=m\hspace{.5mm} p^i; \ i\geq0$, is a positive integer such that $m\nmid q-1$ and $\mathrm{gcd}(m,p)=1.$ Let $e(>2)$ denote the order of $q\hspace{.5mm} (\mathrm{mod}\hspace{.5mm} m)$, then {\upshape Theorem \ref{T4.2}} with $d=q^n-1$ and $g$ as the product of all irreducible factors of $x^{m}-1$ of degree $<e$ gives $\mathcal{S}\leq 2m.$
\end{lemma}
\begin{proof}
	Let $x^{m}-1=g(x)G(x)$, where $g(x)$ is the product of irreducible factors of $x^{m}-1$ of degree less than $e$ and $G(x)$ is the product of irreducible factors of $x^{m}-1$ of degree $e.$ Let $N_0$ and $M_0$ be the number of irreducible factors of $x^{m}-1$ of degree $<e$ and of degree $e$ respectively. Then comparing degrees we get, $m\geq N_0+eM_0$ and hence, $M_0\leq (m-N_0)/e=m(1-\rho(q,m))/e<m/e$, where $\rho(q,m)=N_0/m$ as defined in Section \ref{Sec2}. By Theorem \ref{T4.2} , we get $\mathcal{D}=1-2M_0/q^e> 1-2m/eq^e\geq 1-2/e>0$, since $m\mid q^e-1$ and $\mathcal{S}=\frac{2M_0-1}{\mathcal{D}}+2\leq \frac{2m/e-1}{1-2/e}+2\leq 2m.$ 
\end{proof}
{\small
\begin{longtable}[c]{| c | c | c | c | c | c |c | }
	
	\caption{Pairs $(q,n)\in Q_2(2,2)$ by Sieving technique.\label{Table2}}\\
	\hline
	Sr.No.&$(q,n)$&$d$&$g$&$\mathcal{D}$&$\mathcal{S}$\\
	
	\hline
	\endfirsthead
	\multicolumn{7}{l}{Continuation of Table \ref{Table2}}\\
	\hline
	Sr.No.&$(q,n)$&$d$&$g$&$\mathcal{D}$&$\mathcal{S}$\\
	\hline
	\endhead
	
	\hline
	\endfoot
	
	\hline
	\endfoot
	1&(2, 17)&1 &$ x + 1 $& 0.9844 & 7.0794\\
	2&(2, 22)&1 &$ x + 1 $& 0.2190 & 43.0915\\
	3&(2, 23)&1 &$ x + 1 $& 0.9555 & 9.3261\\
	4&(2, 25)&1 &$ x + 1 $& 0.8060 & 13.1656\\
	5&(2, 27)&1 &$ x + 1 $& 0.1556 & 72.6835\\
	6&(2, 28)&15 &$ x + 1 $& 0.3511 & 33.3323\\
	7&(2, 31)&1 &$ x + 1 $& 0.6249 & 22.8000\\
	8&(2, 33)&1 &$ x + 1 $& 0.0990 & 153.5235\\
	9&(2, 35)&1 &$ x + 1 $& 0.2656 & 66.0123\\
	10&(2, 36)&105 &$ x + 1 $& 0.1098 & 120.3532\\
	11&(2, 39)&1 &$ x + 1 $& 0.1872& 82.1095\\
	12&(2, 40)&3 &$ x + 1 $& 0.0622 & 210.9840\\
	13&(2, 42)&3 &$ x^3 + 1 $& 0.0832 & 206.2723\\
	14&(2, 45)&217 &$ x + 1 $& 0.0489 & 431.6607\\
	15&(2, 48)&105& $ x+1 $&0.1888& 70.8429\\
	16&(2, 51)&1 &$ x + 1 $& 0.1469 & 158.6066\\
	17&(2, 60)&465465 &$ x + 1 $& 0.0226 & 753.1723\\
	18&(2, 63)&7 &$ x^3 + 1 $& 0.1696 & 184.7342\\
	19&(2, 105)&15407 &$ x^3 + 1 $& 0.0245 & 1592.1594\\
	20&(4, 11)&3 &$ 1 $& 0.3837 & 30.6655\\
	21&(4, 13)&3 &$ 1 $& 0.4980 & 20.0706\\
	22&(4, 14)&15 &$ 1 $& 0.2886 & 47.0488\\
	23&(4, 17)&3 &$ 1 $& 0.4687 & 29.7369\\
	24&(4, 20)&15 &$ x + 1 $& 0.3372 & 40.5521\\
	25&(4, 21)&21 &$ x^2 + x + 1 $& 0.2439 & 88.0880\\
	26&(4, 22)&15 &$ 1 $& 0.3778 & 41.7085\\
	27&(4, 24)&105 &$ x^2 + x + 1 $& 0.1888 & 70.8429\\
	28&(4, 25)&33 &$ 1 $& 0.1726& 112.0939\\
	29&(4, 27)&3 &$ x^2 + x + 1 $& 0.0191 & 997.8337\\
	30&(4, 30)&15015&$x^3+1$&0.0831&278.7245\\
	31&(4, 33)&3 &$ x^2 + x + 1 $& 0.0603 & 450.0562\\
	32&(4, 35)&33 &$ 1 $& 0.0234 & 1324.2348\\
	33&(4, 36)&1365 &$ x^2 + x + 1 $& 0.1018 & 208.2453\\
	34&(4, 42)&105 &$ x^2 + x + 1 $& 0.0019 & 16423.5219\\
	35&(4, 45)&4389 &$ x^3 + 1 $& 0.0711 & 522.4620\\
	36&(4, 51)&3 &$ x^2 + x + 1 $& 0.0924 & 467.2519\\
	37&(4, 63)&21 &$ x^3 + 1 $& 0.1738 & 330.0547\\
	38&(8, 5)&1 &$ 1 $& 0.3860 & 25.3139\\
	39&(8, 9)&1 &$ x + 1 $& 0.5619 & 25.1366\\
	40&(8, 10)&3 &$ x + 1 $& 0.4481 & 26.5439\\
	41&(8, 12)&15 &$ x + 1 $& 0.3241 & 42.1078\\
	42&(8, 14)&21 &$ x^4 + x^2 + x + 1 $& 0.1814 & 73.6504\\
	43&(8, 15)&1 &$ x + 1 $& 0.5732 & 35.1497\\
	44&(8, 21)&511 &$ x^4 + x^2 + x + 1 $& 0.0095 & 2831.4159\\
	45&(16, 7)&3 &$ x + 1 $& 0.4501 & 30.8825\\
	46&(16, 9)&15 &$ x + 1 $& 0.1044 & 183.9925\\
	47&(16, 10)&15 &$ x + 1 $& 0.0872 & 196.9413\\
	48&(16, 11)&3 &$ x + 1 $& 0.4817 & 33.1426\\
	49&(16, 12)&15 &$ x + 1 $& 0.1531 & 113.0231\\
	50&(16, 13)&3 &$ x + 1 $& 0.5454 & 36.8396\\
	51&(16, 17)&3 &$ x + 1 $& 0.5207 & 53.8567\\
	52&(16, 18)&105 &$ x + 1 $& 0.2595 & 98.3310\\
	53&(16, 21)&15 &$ x + 1 $& 0.1507 & 234.1820\\
	54&(16, 30)&15 &$ x^{15} + 1 $& 0.0858 & 293.5170\\
	55&(16, 45)&105 &$ x^{15} + 1 $& 0.2732 & 203.3453\\
	56&(32, 3)&1 &$ 1 $& 0.5721 & 17.7323\\
	57&(32, 4)&3 &$ 1 $& 0.2424 & 39.1310\\
	58&(32, 5)&1 &$ x + 1 $& 0.9310 & 9.5184\\
	59&(32, 6)&3 &$ x + 1 $& 0.4467 & 26.6244\\
	60&(32, 31)&1 &$ x^{31} + 1 $& 0.9289 & 15.9958\\
	61&(64, 4)&15 &$ 1 $& 0.4032 & 24.3190\\
	62&(64, 5)&3 &$ x + 1 $& 0.4477 & 31.0381\\
	63&(64, 6)&15 &$ x + 1 $& 0.2929 & 53.2161\\
	64&(64, 7)&3 &$ x + 1 $& 0.4582 & 47.8293\\
	65&(64, 9)&3 &$ x + 1 $& 0.3316 & 77.3932\\
	66&(64, 10)&15 &$ x + 1 $& 0.2108 & 101.6394\\
	67&(64, 21)&21 &$ x + 1 $& 0.1738 & 330.0547\\
	68&($2^7$, 3)&1 &$ x + 1 $& 0.6925 & 12.1086\\
	99&($2^7$, 4)&3 &$ x + 1 $& 0.4511 & 21.9523\\
	70&($2^8$, 3)&15&$ 1 $&0.4111&33.6258\\
	71&($2^8$, 5)&3 &$ x + 1 $& 0.1560 & 123.8295\\
	72&($2^8$, 6)&15 &$ x + 1 $& 0.3875 & 45.8714\\
	73&($2^8$, 9)&15 &$ x + 1 $& 0.2091 & 131.1033\\
	74&($2^8$, 15)&105&$x+1$ &0.2621& 196.5840\\
	75&($2^9$, 4)&15 &$ x + 1 $& 0.3554 & 32.9531\\
	76&($2^{10}$, 3)&3 &$ x + 1 $& 0.4448 & 31.2294\\
	77&($2^{10}$, 6)&15 &$ x + 1 $& 0.2078 & 103.0440\\
	78&($2^{12}$, 3)&15 &$ x + 1 $& 0.3544 & 44.3251\\
	79&($2^{12}$, 5)&15 &$ x + 1 $& 0.2098 & 121.1705\\
	80&($2^{16}$, 3)&3 &$ x + 1 $& 0.0031 & 6210.5717\\
	81&($2^{20}$, 3)&15 &$ x + 1 $& 0.2117 & 101.1816\\
\end{longtable}
}

%Lemma5.4------------------------------------
\begin{lemma}\label{l5.4}
	Let $q=2^k$ where $1\leq k\leq5$ , $n\geq 3
	$ and $m\nmid q^2-1$. Then $(q,n)\in Q_2(2,2)$ unless
	\begin{enumerate}[$(i)$]
		\item $q=2$ and $n= 5, 7, 9, 10, 11, 13, 14, 15, 18, 20, 30.$
		\item $q=4$ and $n=7, 9, 18.$
		
	\end{enumerate}	
\end{lemma}
\begin{proof}
	Since $m\nmid q^2-1$, $e>2$ (defined in above Lemma) and $m\nmid q-1.$ Let $d=q^n-1$ and $g$ be the product of all irreducible factors of $x^{m}-1$ of degree less than $e$, then $S\leq 2m$ by above lemma. Hence, Inequality \eqref{Eq5} transformed into 
	\begin{equation}\label{Eq8}
	q^{n/2}>7 \C^2  q^{2n/\nu} 2^{2m\rho(q,m)} 2m
	\end{equation}
	{${\textbf{Case} \ \bm {k=1:}}$}  From Lemma \ref{l2.5}, we have $\rho(2,m)\leq
	1/6$ ($m\neq 5,9,21$). Since $m\leq n$, Inequality \eqref{Eq8} holds if  
	$$q^{n/2}>7 \C^2 q^{2n/\nu} 2^{n/3} 2n$$ holds and the latter inequality is true for $n\geq19340$ and $\nu=13.7.$ For $n<19340$, by directly verifying the inequality $q^{n/2}>7W(q^n-1)^2W(x^n-1)^2$, we get $(2,n)\in Q_2(2,2)$ unless $n=$ 7, 11, 13, 14, 15, 17, 22, 23, 25, 27, 28, 30, 31, 33, 35, 39, 45, 51, 60, 63, 105.  For these values of $n$, we choose the values of $d$ and $g$ listed in Table \ref{Table2} such that the inequality $q^{n/2}>7W(d)^2W(g)^2\mathcal{S}$ holds and get that $(2,n)\in Q_2(2,2)$ unless $n=$ 7, 11, 13, 14, 15, 30. 
	
	Let $m=5$, then $\rho(2,m)=1/5$ by Lemma \ref{l2.5}, and from Inequality \eqref{Eq8}, we get 
	$$q^{5\cdot2^i/2}>240\hspace{.5mm} \C^2 q^{10\cdot2^i/\nu},$$
	which is true for $i\geq5.$ Hence,  $(2,n)\in Q_2(2,2)$ unless $n=$ 5, 10, 20, 40, 80. For these values of $n$, we choose the values of $d$ and $g$ listed in Table \ref{Table2} such that the inequality $q^{n/2}>6W(d)^2W(g)^2\mathcal{S}$ holds and get that $(2,n)\in Q_2(2,2)$ unless $n=$ 5, 10, 20.
	
	Now, let $m=9$, then $\rho(2,m)=2/9$ by Lemma \ref{l2.5}, and from Inequality \eqref{Eq8}, we get 
	$$q^{9\cdot2^i/2}>1728\hspace{.5mm} \C^2 q^{18\cdot2^i/\nu},$$
	which is true for $i\geq4.$ Hence,  $(2,n)\in Q_2(2,2)$ unless $n=$ 9, 18, 36, 72. For these values of $n$, we choose the values of $d$ and $g$ listed in Table \ref{Table2} such that the inequality $q^{n/2}>6W(d)^2W(g)^2\mathcal{S}$ holds and get that $(2,n)\in Q_2(2,2)$ unless $n=$ 9, 18.
	
	Finally, let $m=21$, then $\rho(2,m)=4/21$  by Lemma \ref{l2.5}, and from Inequality \eqref{Eq8}, we get  
	$$q^{21\cdot2^i/2}>64512\hspace{.5mm} \C^2 q^{42\cdot2^i/\nu}$$
	which is true for $i\geq3.$ Hence,  $(2,n)\in Q_2(2,2)$ unless $n=$ 21, 42, 84. For these values of $n$, we choose the values of $d$ and $g$ listed in Table \ref{Table2} such that the inequality $q^{n/2}>6W(d)^2W(g)^2\mathcal{S}$ holds and get that $(2,n)\in Q_2(2,2)$ for all $n$.\\
	
	\noindent
	{${\textbf{Case} \ \bm {k=2:}}$} From Lemma \ref{l2.5}, we have $\rho(4,m)\leq
	1/5$ ($m\neq 9,45$), and Inequality \eqref{Eq8} holds if 
	$$q^{n/2}>6\hspace{.5mm} \C^2 q^{2n/\nu} 2^{2n/5} 2n$$
	holds, which is true for $n\geq308.$ For $n<308$,  by directly verifying the inequality $q^{n/2}>6W(q^n-1)^2W(x^n-1)^2$ we get $(4,n)\in Q_2(2,2)$ unless $n=$ 7, 11, 13, 14, 17, 21, 22, 25, 27, 33, 35, 42, 51, 63. For these values of $n$, we choose the values of $d$ and $g$ listed in Table \ref{Table2} such that the inequality $q^{n/2}>6W(d)^2W(g)^2\mathcal{S}$ holds and get that $(4,n)\in Q_2(2,2)$ unless $n=$ 7.

	Let $m=9$, then $\rho(4,m)=1/3$  by Lemma \ref{l2.5}, and from Inequality \eqref{Eq8}, we get  
	$$q^{9\cdot2^i/2}>6912\hspace{.5mm} \C^2 q^{18\cdot2^i/\nu},$$
	which is true for $i\geq3.$ Hence,  $(4,n)\in Q_2(2,2)$ unless $n=$ 9, 18, 36. For these values of $n$, we choose the values of $d$ and $g$ listed in Table \ref{Table2} such that the inequality $q^{n/2}>6W(d)^2W(g)^2\mathcal{S}$ holds and get that $(4,n)\in Q_2(2,2)$ unless $n=$ 9, 18.

	Now, let $m=45$, then $\rho(4,m)=11/45$ by Lemma \ref{l2.5}, and from Inequality \eqref{Eq8}, we get  
	$$q^{45\cdot2^i/2}>2264924160\hspace{.5mm} \C^2 q^{90\cdot2^i/\nu},$$
	which is true for $i\geq2.$ Hence,  $(4,n)\in Q_2(2,2)$ unless $n=$ 45, 90. For these values of $n$, we choose the values of $d$ and $g$ listed in Table \ref{Table2} such that the inequality $q^{n/2}>6W(d)^2W(g)^2\mathcal{S}$ holds and get that $(4,n)\in Q_2(2,2)$ for all $n$.\\
	
	\noindent
	{${\textbf{Case} \ \bm {k=3:}}$} From Lemma \ref{l2.5}, we have $\rho(8,m)\leq
	1/3$, and Inequality \eqref{Eq8} holds if 
	$$q^{n/2}>6\hspace{.5mm} \C^2 q^{2n/\nu} 2^{2n/3} 2n$$
	holds, which is true for $n\geq284.$ For $n<284$,  by directly verifying the inequality $q^{n/2}>6W(q^n-1)^2W(x^n-1)^2$ we get $(8,n)\in Q_2(2,2)$ unless $n=$ 5, 10, 15. For these values of $n$, we choose the values of $d$ and $g$ listed in Table \ref{Table2} such that the inequality $q^{n/2}>6W(d)^2W(g)^2\mathcal{S}$ holds and get that $(8,n)\in Q_2(2,2)$ for all $n$.\\ 
	
	\noindent
	{${\textbf{Case} \ \bm {k=4:}}$} From Lemma \ref{l2.5}, we have $\rho(16,m)\leq
	1/3$, and Inequality \eqref{Eq8} holds if 
	$$q^{n/2}>6\hspace{.5mm} \C^2 q^{2n/\nu} 2^{2n/3} 2n$$
	holds, which is true for $n\geq98.$ For $n<98$,  by directly verifying the inequality $q^{n/2}>6W(q^n-1)^2W(x^n-1)^2$ we get $(16,n)\in Q_2(2,2)$ unless $n=$ 7, 9, 11, 13, 18, 21, 45. For these values of $n$, we choose the values of $d$ and $g$ listed in Table \ref{Table2} such that the inequality $q^{n/2}>6W(d)^2W(g)^2\mathcal{S}$ holds and get that $(16,n)\in Q_2(2,2)$ for all $n$.\\
	
	\noindent
	{${\textbf{Case} \ \bm {k=5:}}$} From Lemma \ref{l2.5}, we have $\rho(32,m)\leq
	1/3$, and Inequality \eqref{Eq8} holds if 
	$$q^{n/2}>6\hspace{.5mm} \C^2 q^{2n/\nu} 2^{2n/3} 2n$$
	holds, which is true for $n\geq55.$ For $n<55$,  by directly verifying the inequality $q^{n/2}>6W(q^n-1)^2W(x^n-1)^2$ we get $(32,n)\in Q_2(2,2)$ unless $n=$ 5. For this $n$, we choose the values of $d$ and $g$ listed in Table \ref{Table2} such that the inequality $q^{n/2}>6W(d)^2W(g)^2\mathcal{S}$ holds and get that $(32,n)\in Q_2(2,2)$ for all $n$.
\end{proof}

In the above discussion, we have found at most 41 exceptional fields $\mathbb{F}_{q^n}$ in which a primitive normal pair $(\alpha,f(\alpha))$ may not exist for some $f\in R_{2^k}(2,2)$ and $n\geq3$. Clearly, for $n=1$ and 2, every primitive element in $\mathbb{F}_{q^n}$ is $(x^n-1)$-free, so it is normal as well over $\mathbb{F}_{q}.$ Hence, for $n=1,2$, a pair $(\alpha,f(\alpha))$ in $\mathbb{F}_{q^n}$ is primitive if and only if it is primitive normal over $\mathbb{F}_{q}.$ Existence of such primitive pairs has already been  settled by R. K. Sharma et al. in \cite{Awasthi}. Hence, from their results and our discussion on primitive pairs in  Section \ref{Sec3}, we get that $(q,n)\in Q_2(2,2)$ unless
\begin{enumerate}[(i)]
	\item $q=$ 2, 4, 8, 16, 64, 256, 1024, 4096 and $n=1.$
	\item $q=$ 2, 4, 8, 16, 32, 64 and $n=2.$
\end{enumerate} 
Finally, with the help of the above discussion and Lemmas \ref{l5.1}, \ref{l5.2}, \ref{l5.4}, we conclude the following  
\begin{theorem}
	Let $ q=2^k$ and $n\geq1$ be a positive integer. Then $ (q,n)\in Q_2(2,2)$ unless
	\begin{enumerate}[$(i)$]
		\item $q=2$ and $n=1,2,3,4,5,6,7,8, 9,10,11,12,13,14,15,16,18,20,24,30.$
		\item $q=4$ and $n=1,2,3,4,5,6,7,8,9,10,12,15,18.$
		\item $q=8$ and $n=1,2,3,4,6,7,8, 14.$
		\item $q=16$ and $n=1,2,3,4,5,6,15.$
		\item $q=32$ and $n=2.$
		\item $q=64$ and $n=1,2,3.$
		\item $q=256,1024,4096$ and $n=1.$

	\end{enumerate}
\end{theorem}
For these $55$ possible exceptions, we perform some experiments and find that for some pairs $(q,n),$ there exist at least one rational function $f\in R_{q^n}(2,2)$ such that $f(\alpha)$ is not a primitive normal element in $\mathbb{F}_{q^n}$ for any primitive normal element $\alpha \in \mathbb{F}_{q^n}.$ We list these genuine exceptional pairs $(q,n)$ in Table \ref{Table3}, where $a$ and $b$ are the generators of multiplicative cyclic groups $\mathbb{F}_{q}^*$ and $\mathbb{F}_{q^n}^*$ respectively.

{\small
	
\begin{longtable}[h]{|c|c|l|c|}
	\caption{Pairs $(q,n)\not \in Q_2(2,2)$ \label{Table3}}\\
	\hline
	S.No.&$(q,n)$&Primitive Normal Elements of $\mathbb{F}_{q^n}$ over $\mathbb{F}_{q}$&$f$\\
	\hline
	\endfirsthead
	\multicolumn{4}{l}{Continuation of Table \ref{Table3}}\\
	\hline
	S.No.&$(q,n)$&Primitive Normal Elements&$f$\\
	\hline
	\endhead
	
	\hline
	\endfoot
	
	\hline
	\endfoot
	
	$1$&$(2,1)$&$\{b^i;i=1\},\ b+1=0$&$\frac{x^2+1}{x^2+x+1}$\\
	$2$&$(2,2)$&$\{b^i;i=1,2\},\ b^2+b+1=0$&$\frac{bx^2}{x^2+b}$\\
	$3$&$(2,3)$&$\{b^i;i=3,5,6\},\ b^3+b+1=0$&$\frac{bx^2}{x^2+b}$\\
	$4$&$(2,4)$&$\{b^i;i=7,11,13,14\},\ b^4+b+1=0$&$\frac{bx^2}{x^2+b}$\\
	$5$&$(2,5)$&$\{b^i;i=3, 5, 6, 9, 10, 11, 12, 13, 17, 18, 20,$&$\frac{bx^2+bx+b^2}{x^2+1}$\\
	&&$ 21, 22, 24, 26\},b^5+b^2+1=0$&\\
	$6$&$(2,6)$&$\{b^i;i=5, 10, 17, 20, 23, 29, 31, 34, 40, 43,$&$\frac{bx^2+b^4}{x^2+b^4x+b^2}$\\
	&&$ 46, 47, 53, 55, 58, 59, 61, 62\},b^6+b+1=0$&\\
	$7$&$(4,1)$&$\{b^i;i=1,2\},\ b+a=0,a^2+a+1=0$&$\frac{ax^2}{x^2+a}$\\
	$8$&$(4,2)$&$\{b^i;i=1, 2, 4, 7, 8, 11, 13, 14\},$&$\frac{ax^2+a}{x^2+ax}$\\
	&&$b^2+ab+a=0,a^2+a+1=0$&\\
	$9$&$(4,3)$&$\{b^i;i=1, 2, 4, 8, 11, 16, 22, 23, 25, 29, 32, 37,$&$\frac{ax^2}{x^2+ax+a}$\\
	&&$ 43, 44, 46, 50, 53, 58\},b^3+ab^2+ab+a=0,$&\\
	&&$a^2+a+1=0$&\\
	$10$&$(8,1)$&$\{b^i;i=1, 2, 3, 4, 5, 6\},\ b+a=0, $&$\frac{ax^2+ax+a+1}{x^2+x+a}$\\
	&&$a^3+a+1=0$&\\
	$11$&$(16,1)$&$\{b^i;i=1, 2, 4, 7, 8, 11, 13, 14\},\ b+a=0, $&$\frac{ax^2+a^2}{x^2+x+a^2}$\\
	&&$a^4+a+1=0$&\\
	\hline
\end{longtable}}

For the remaining pairs $(q,n),$ we are unable to exhaust the complete space $R_{q^n}(2,2)$ due to limited computational resources. Hence, for the remaining $44$ pairs $(q,n)$ the existence of a primitive normal pair $(\alpha,f(\alpha))$ for each $f\in R_{q^n}(2,2)$ is not guaranteed. 

\section{Conclusion}\label{Sec6}
In this paper, we have provided a sufficient condition for the existence of primitive normal pairs $(\alpha,f(\alpha))$ in $\mathbb{F}_{q^n}$ over $\mathbb{F}_{q}$ for every rational function $f\in R_{q^n}(m_1,m_2)$ and found that there may exist some rational functions $f\in R_{q^n}(2,2)$ for which such a pair may not exist in at most $55$ finite fields $\mathbb{F}_{q^n}$ for $q=2^k.$
 
\section{Acknowledgements}
We are grateful to the anonymous reviewers for their valuable comments and suggestions.

This research work is supported by UGC, under Grant Ref. No. 1042/ CSIR-UGC NET DEC-2018 and  CSIR, under Grant
F. No. 09/045(1674)/2019-EMR-I.

\bibliographystyle{unsrt}
\bibliography{Normal_Pairs_2.bib}

\begin{thebibliography}{10}

\bibitem{mullin}
Ronald~C Mullin, Ivan~M Onyszchuk, Scott~A Vanstone, and Richard~M Wilson.
\newblock Optimal normal bases in gf (pn).
\newblock {\em Discrete applied mathematics}, 22(2):149--161, 1988.

\bibitem{Omura}
Jimmy~K. Omura and James~L. Massey.
\newblock Computational method and apparatus for finite field arithmetic, May~6
  1986.
\newblock US Patent 4,587,627.

\bibitem{Onyszchuk}
Ivan~M. Onyszchuk, Ronald~C. Mullin, and Scott~A. Vanstone.
\newblock Computational method and apparatus for finite field multiplication,
  May~17 1988.
\newblock US Patent 4,745,568.

\bibitem{diffie}
Whitfield Diffie and Martin Hellman.
\newblock New directions in cryptography.
\newblock {\em IEEE transactions on Information Theory}, 22(6):644--654, 1976.

\bibitem{blum}
Manuel Blum and Silvio Micali.
\newblock How to generate cryptographically strong sequences of pseudorandom
  bits.
\newblock {\em SIAM journal on Computing}, 13(4):850--864, 1984.

\bibitem{Agnew}
G.~B. Agnew, R.~C. Mullin, I.~M. Onyszchuk, and S.~A. Vanstone.
\newblock An implementation for a fast public-key cryptosystem.
\newblock {\em J. Cryptology}, 3(2):63--79, 1991.

\bibitem{CarlPrim}
L.~Carlitz.
\newblock Primitive roots in a finite field.
\newblock {\em Trans. Amer. Math. Soc.}, 73:373--382, 1952.

\bibitem{CarlSome}
L.~Carlitz.
\newblock Some problems involving primitive roots in a finite field.
\newblock {\em Proc. Nat. Acad. Sci. U.S.A.}, 38:314--318; errata, 618, 1952.

\bibitem{Daven}
H.~Davenport.
\newblock Bases for finite fields.
\newblock {\em J. London Math. Soc.}, 43:21--39, 1968.

\bibitem{Lenstra}
H.~W. Lenstra, Jr. and R.~J. Schoof.
\newblock Primitive normal bases for finite fields.
\newblock {\em Math. Comp.}, 48(177):217--231, 1987.

\bibitem{TianQi}
Tian Tian and Wen~Feng Qi.
\newblock Primitive normal element and its inverse in finite fields.
\newblock {\em Acta Math. Sinica (Chin. Ser.)}, 49(3):657--668, 2006.

\bibitem{CoHuc}
Stephen~D. Cohen and Sophie Huczynska.
\newblock The strong primitive normal basis theorem.
\newblock {\em Acta Arith.}, 143(4):299--332, 2010.

\bibitem{CohenConsecutive}
Stephen~D. Cohen.
\newblock Consecutive primitive roots in a finite field.
\newblock {\em Proc. Amer. Math. Soc.}, 93(2):189--197, 1985.

\bibitem{SDEven}
Stephen~D. Cohen.
\newblock Pairs of primitive elements in fields of even order.
\newblock {\em Finite Fields Appl.}, 28:22--42, 2014.

\bibitem{Wang}
Peipei Wang, Xiwang Cao, and Rongquan Feng.
\newblock On the existence of some specific elements in finite fields of
  characteristic 2.
\newblock {\em Finite Fields Appl.}, 18(4):800--813, 2012.

\bibitem{Liao}
Qunying Liao, Jiyou Li, and Keli Pu.
\newblock On the existence for some special primitive elements in finite
  fields.
\newblock {\em Chinese Annals of Mathematics, Series B}, 37(2):259--266, 2016.

\bibitem{KapeNBT}
Giorgos Kapetanakis.
\newblock An extension of the (strong) primitive normal basis theorem.
\newblock {\em Appl. Algebra Engrg. Comm. Comput.}, 25(5):311--337, 2014.

\bibitem{Kape}
Giorgos Kapetanakis.
\newblock Normal bases and primitive elements over finite fields.
\newblock {\em Finite Fields Appl.}, 26:123--143, 2014.

\bibitem{ARKS}
Anju and R.~K. Sharma.
\newblock On primitive normal elements over finite fields.
\newblock {\em Asian-Eur. J. Math.}, 11(2):1850031, 14, 2018.

\bibitem{Booker}
Andrew~R. Booker, Stephen~D. Cohen, Nicole Sutherland, and Tim Trudgian.
\newblock Primitive values of quadratic polynomials in a finite field.
\newblock {\em Math. Comp.}, 88(318):1903--1912, 2019.

\bibitem{AnjuPN}
Anju and R.~K. Sharma.
\newblock Existence of some special primitive normal elements over finite
  fields.
\newblock {\em Finite Fields Appl.}, 46:280--303, 2017.

\bibitem{Awasthi}
Rajendra~K. Sharma, Ambrish Awasthi, and Anju Gupta.
\newblock Existence of pair of primitive elements over finite fields of
  characteristic 2.
\newblock {\em J. Number Theory}, 193:386--394, 2018.

\bibitem{hari}
Stephen~D Cohen, Hariom Sharma, and Rajendra Sharma.
\newblock Primitive values of rational functions at primitive elements of a
  finite field.
\newblock {\em Journal of Number Theory}, 219:237--246, 2021.

\bibitem{CarNeu}
C{\'\i}cero Carvalho, Jo{\~a}o~Paulo Guardieiro, Victor~GL Neumann, and
  Guilherme Tizziotti.
\newblock On special pairs of primitive elements over a finite field.
\newblock {\em Finite Fields and Their Applications}, 73:101839, 2021.

\bibitem{carvalho2}
C{\'\i}cero Carvalho, Jo{\~a}o~Paulo Guardieiro, Victor~GL Neumann, and
  Guilherme Tizziotti.
\newblock On the existence of pairs of primitive and normal elements over
  finite fields.
\newblock {\em Bulletin of the Brazilian Mathematical Society, New Series},
  pages 1--23, 2021.

\bibitem{Fu}
L.~{Fu} and D.~{Wan}.
\newblock A class of incomplete character sums.
\newblock {\em Quarterly Journal of Mathematics}, 65(4):1195--1211, 2014.

\bibitem{Nieder}
Rudolf Lidl and Harald Niederreiter.
\newblock {\em Finite fields}, volume~20.
\newblock Cambridge University Press, Cambridge, second edition, 1997.

\bibitem{rani}
Mamta Rani, Avnish~K Sharma, Sharwan~K Tiwari, and Indivar Gupta.
\newblock On the existence of pairs of primitive normal elements over finite
  fields.
\newblock {\em S{\~a}o Paulo Journal of Mathematical Sciences}, pages 1--18,
  2021.

\bibitem{sagemath}
{The Sage Developers}.
\newblock {\em {S}ageMath, the {S}age {M}athematics {S}oftware {S}ystem
  ({V}ersion 9.0)}, 2020.

\end{thebibliography}
\end{document}